\documentclass[12pt, reqno]{amsart}

\usepackage{amsfonts}
\usepackage{amssymb}
\usepackage{amsthm}
\usepackage{amsmath}
\usepackage{a4wide}
\usepackage{mathrsfs}
\usepackage{epsfig}
\usepackage{palatino}
\usepackage{tikz,fp,ifthen}
\usepackage{float}
\usepackage{graphicx,cite}
\usepackage{comment}
\usepackage{accents}

\usepackage{graphicx}

\usepackage{color}
\definecolor{citegreen}{rgb}{0,0.6,0}
\definecolor{refred}{rgb}{0.8,0,0}
\usepackage[colorlinks, citecolor=citegreen,linkcolor=refred]{hyperref}

\usepackage{mathtools}
\mathtoolsset{showonlyrefs}

\usepackage[]{draftwatermark}
\SetWatermarkText{${\mathcal{DRAFT}}$}
\SetWatermarkScale{6.2}

\newtheorem{thm}{Theorem}[section]
\newtheorem{lem}[thm]{Lemma}
\newtheorem{prop}[thm]{Proposition}

\theoremstyle{definition}
\newtheorem{defn}[thm]{Definition}

\theoremstyle{remark}
\newtheorem{rem}[thm]{Remark}

\numberwithin{equation}{section}

\def\eps{\varepsilon}

\def\Ric{{\mathrm {Ric}}}

\def\R{\mathbb R}

\def\R{{{\mathbb R}}}

\def \ringg#1{\accentset{\circ}{#1}}

\newcommand{\intbar}{\etaathop{\int\etaakebox(-13.5,0){\rule[4pt]{.7em}{0.3pt}}
\kern-6pt}\nolimits}

\newcommand{\be}{\begin{equation}}
\newcommand{\ee}{\end{equation}}
\newcommand{\bea}{\begin{equation*}}
\newcommand{\eea}{\end{equation*}}

\usetikzlibrary{backgrounds}
\usetikzlibrary{decorations.pathmorphing,backgrounds,fit,calc,through}
\usetikzlibrary{arrows}
\usetikzlibrary{shapes,decorations,shadows}
\usetikzlibrary{fadings}
\usetikzlibrary{patterns}
\usetikzlibrary{mindmap}
\usetikzlibrary{decorations.text}
\usetikzlibrary{decorations.shapes}

\begin{document}

\title{A matrix Harnack inequality for semilinear heat equations}

\author{Giacomo Ascione}
\address{Giacomo Ascione\\
Dipartimento di Matematica e Applicazioni,
Universit\`a di Napoli, Via Cintia, Monte S. Angelo 80126 Napoli,
Italy}
\email{giacomo.ascione@unina.it}

\author{Daniele Castorina}
\address{Daniele Castorina\\
Dipartimento di Matematica e Applicazioni,
Universit\`a di Napoli, Via Cintia, Monte S. Angelo 80126 Napoli,
Italy}
\email{daniele.castorina@unina.it}

\author{Giovanni Catino}
\address{Giovanni Catino\\
Dipartimento di Matematica,
Politecnico di Milano, Piazza Leonardo da Vinci 32, 20133, Milano,
Italy}
\email{giovanni.catino@polimi.it}

\author{Carlo Mantegazza}
\address{Carlo Mantegazza\\
Dipartimento di Matematica e Applicazioni,
Universit\`a di Napoli, Via Cintia, Monte S. Angelo 80126 Napoli,
Italy}
\email{c.mantegazza@sns.it}

\begin{abstract}  
We derive a matrix version of Li \& Yau--type estimates for positive solutions of semilinear heat equations on Riemannian manifolds with nonnegative sectional curvatures and parallel Ricci tensor, similarly to what R.~Hamilton did in~\cite{hamilton7} for the standard heat equation. We then apply these estimates to obtain some Harnack--type inequalities, which give local bounds on the solutions in terms of the geometric quantities involved.
\end{abstract}


\maketitle

\section{Introduction}

We are interested in positive classical solutions of semilinear heat equations $u_t = \Delta u + u^p$, in $\R^n$ or in a complete Riemannian manifold $(M,g)$ without boundary, where $p>1$.

In their celebrated paper~\cite{liyau}, Li and Yau showed how a Harnack inequality for the classical heat equation on a manifold with nonnegative Ricci tensor can be derived from a differential inequality for the logarithm of a solution. Subsequently, in the case of a manifold with nonnegative sectional curvatures and parallel Ricci tensor, Hamilton in~\cite{hamilton7} proved that the Harnack estimate of Li and Yau can actually be obtained as the trace of a full matrix inequality, under some more restrictive geometric assumptions. In some cases these inequalities can be useful in proving triviality of eternal solutions, see~\cite{cacama2}, moreover, ``geometric'' versions of them appear naturally and play a key role in the analysis of mean curvature flow and of Ricci flow (which are described by much more complicated systems of parabolic PDEs), see~\cite{hamilton12,hamilton4}. 

Our aim is to extend to the semilinear setting the matrix Harnack estimate of Li \& Yau--type for the heat equation developed by Hamilton in~\cite{hamilton7}. 

\medskip

We set some definitions and notations. In all of the paper, the Riemannian manifolds $(M,g)$ will be smooth, complete, connected and without boundary. We will denote with $\nabla$ the Levi--Civita connection of $(M,g)$ and $\Delta$ the associated Laplace--Beltrami operator and we assume that $(M,g)$ has nonnegative sectional curvatures and parallel Ricci tensor, that is, $\nabla \Ric = 0$. Finally, all the solutions we will consider are classical ($C^2$ in space and $C^1$ in time, at least).

\begin{rem}
The hypothesis that $M$ has parallel Ricci tensor and nonnegative sectional curvatures
is satisfied on a torus or a sphere or a complex projective space, or a product
of such, or a quotient of a product by a finite group of isometries.
\end{rem}

\begin{defn}\label{def:adm}
A quintuple of real numbers $(a,b,c,d,\theta)$ is {\em admissible}, if the following inequalities
\begin{equation}\label{eq:systineq}
\begin{cases}
d\geq a>c>0 \\
\theta>b \ge 0\\
(a-c)^2\theta^2-a(\theta-b)[(2\theta+na)(a-c)+a(n-1)(\theta-b)]\ge 0
\end{cases}
\end{equation}
are satisfied.
\end{defn}

The above system turns out to have actually solutions.

\begin{prop}\label{prop:exist}
There exists a nonempty cone $\mathcal{C}$ of admissible quintuples of parameters.
\end{prop}

We can then state our main result.

\begin{thm}\label{main}
Suppose $(M,g)$ is a complete, $n$--dimensional Riemannian manifold without boundary, with nonnegative sectional curvatures and parallel Ricci tensor. Let $f=\log u$, where $u$ is a positive classical solution of 
\begin{equation}\label{eq:semiheat}
\partial_t u=\Delta u + u^p
\end{equation}
in $M \times (0,T)$. Then, for any $(a,b,c,d,\theta) \in \mathcal{C}$ there exists a constant $\varepsilon>0$ such that 
\begin{equation}\label{MAIN}
t \left( \theta f_{ij} + a \Delta f g_{ij} + b f_i f_j + c  |\nabla f|^2 g_{ij} + d e^{(p-1)f} g_{ij} \right) \geq -\frac{1}{\varepsilon}g_{ij},
\end{equation}
in the sense of tensors in $M\times(0,T)$, for all $1<p<1+G(a,b,c,d,\theta)$, where 
\begin{equation}\label{def:G}
G(a,b,c,d,\theta)=\min\{G_1(b,d,\theta),G_2(a,b,c,d,\theta)\},
\end{equation}
given
\begin{equation*}
G_1(b,d,\theta)=\frac{4d(\theta-b)}{\theta^2}
\end{equation*}
and $G_2(a,b,c,d,\theta)$ the positive solution of 
\begin{equation}\label{eq:square}
(d-a)\theta^2 x^2 + (d-c) \theta^2 x -4cd(\theta-b)=0.
\end{equation}
\end{thm}

Taking the trace with the metric $g$ in inequality~\eqref{MAIN}, we get a scalar Li \& Yau--type inequality (which is actually weaker than the analogous one proved in~\cite{cacama2})
\begin{equation}\label{MAIN2}
t \left[(\theta +na)\Delta f + (b+nc)|\nabla f|^2 + nde^{(p-1)f}\right] \geq -\frac{n}{\varepsilon}
\end{equation}
and substituting $u=e^f$,
\begin{equation*}
(\theta/n +a)\Delta u + (b/n+c-\theta/n-a)\frac{|\nabla u|^2}{u} + d u^p\geq -\frac{u}{\varepsilon t}
\end{equation*}
in $M\times(0,T)$.

\begin{rem}\label{mainrem}
As we mentioned, inequalities like~\eqref{MAIN} and~\eqref{MAIN2} are relevant for {\em ancient} (and {\em eternal}) solutions $u$, that is, solutions defined in $M\times(-\infty,T)$ for some $T\in\R\cup\{+\infty\}$, since by a standard argument (see~\cite[Section~3]{cacama2}, for instance) choosing suitable intervals for their application, they imply
\begin{equation*}
\theta u_{ij} + a \Delta u g_{ij}+(b-\theta)\frac{u_iu_j}{u} + (c-a) \frac{|\nabla u|^2}{u} g_{ij} + d u^pg_{ij}\geq 0
\end{equation*}
in the sense of tensors and
\begin{equation*}
(\theta/n +a)\Delta u + (b/n+c-\theta/n-a)\frac{|\nabla u|^2}{u} + d u^p\geq 0
\end{equation*}
in $M\times(-\infty,T)$.
\end{rem}

It is possible to have an explicit bound for the function $G$ defined in Theorem~\ref{main} only in terms of the dimension $n$ of the manifold, as shown in the next proposition, hence giving a {\em lower} bound for the range of exponents $p$ for which the results hold.

\begin{prop}\label{prop:lower}
There holds
$$
\sup_{(a,b,c,d,\theta) \in \mathcal{C}} G(a,b,c,d,\theta) \geq \widetilde{G}(n)
$$
where
\begin{equation*}
\widetilde{G}(n) = \frac{4}{(k(n)+1)^2}\Bigl(1+\frac{1}{z(n)}\Bigr),
\end{equation*}
and 
\begin{align*}
k(n)=&\,3\sqrt{n}\cos\Bigl(\frac{1}{3}\arccos\bigl(1/\sqrt{n}\,\bigr)\Bigr),\\
z(n)=&\,\frac{k^2(n)-3n+\sqrt{k^4(n)-6nk^2(n)-6nk(n)}}{3n}.
\end{align*}
In particular, for any $c>0$
$$
\bigl((z(n)+1)c,k(n)c,c,(z(n)+1)c,(k(n)+1)c\bigr) \in \mathcal{C}
$$
and
\begin{equation*}
\widetilde{G}(n)=G((z(n)+1)c,k(n)c,c,(z(n)+1)c,(k(n)+1)c).
\end{equation*}
\end{prop}

In the next sections we show these results, while in the last one we derive, along the lines of~\cite{liyau} (and~\cite{hamilton7}), some consequent Harnack--type local estimates for the solutions.

\section{Proof of Theorem~\ref{main}}\label{pfmain}

We follow the line of Hamilton in~\cite{hamilton7} and, by simplicity and clarity, we show the proof when $(M,g)$ is $\R^n$ with its canonical metric. In the general case of a complete $n$--dimensional Riemannian manifold, the extra curvature terms which appear in the computations because of the operations of interchanging covariant derivatives ``have the right sign'' in the final inequality, since we assumed nonnegative sectional curvatures and parallel Ricci tensor. Furthermore, thanks to standard localization arguments (as explicitly shown in~\cite{cacama2}, see also~\cite{hamilton12}), in applying the maximum principle -- on which the proof is based -- we can argue as if we were in a compact case (all the maximum/minimum points there exist).

Let $u:M\times[0,T)\to\R$ be a positive solution of $u_t = \Delta u + u^p$. Setting 
$f=\log u$, we have
$$
\vert\nabla f\vert=\frac{\vert\nabla u\vert}{u}\qquad\qquad \Delta f=\frac{\Delta u}{u}-\frac{\vert\nabla u\vert^2}{u^2}=\frac{\Delta u}{u}-\vert\nabla f\vert^2
$$
$$
f_t=\frac{u_t}{u}=\frac{\Delta u}{u}+u^{p-1}=\Delta f+\vert\nabla f\vert^2+e^{f(p-1)}.
$$
Moreover, by equation~\eqref{eq:semiheat}, we also have the following relations:
\begin{align*}
(\partial_t - \Delta) f &= |\nabla f|^2 + e^{(p-1)f}\\
(\partial_t - \Delta) f_i &= 2 f_{ik} f_k + (p-1) e^{(p-1)f} f_i\\
(\partial_t - \Delta) (f_i f_j) &= 2 f_{ik} f_k f_j + 2 f_{jk} f_k f_i + 2 (p-1) e^{(p-1)f} f_i f_j - 2 f_{ik} f_{jk}\\
(\partial_t - \Delta) |\nabla f|^2 &= 4 f_{lk} f_l f_k + 2 (p-1) e^{(p-1)f} |\nabla f|^2 - 2 |\nabla^2 f|^2\\
(\partial_t - \Delta) f_{ij} &= 2 f_{ikj} f_k + 2 f_{ik} f_{jk} + (p-1)^2 e^{(p-1)f} f_i f_j + (p-1)e^{(p-1)f} f_{ij}\\
(\partial_t - \Delta) \Delta f &= 2 (\Delta f)_k f_k + 2 |\nabla^2 f|^2 + (p-1)^2 e^{(p-1)f} |\nabla f|^2 + (p-1) e^{(p-1)f} \Delta f\\
(\partial_t - \Delta) e^{(p-1)f} &= 2(p-1) e^{(p-1)f} |\nabla f|^2 - p(p-1) e^{(p-1)f} |\nabla f|^2 + (p-1) e^{2(p-1)f}.
\end{align*}
Let $(a,b,c,d,\theta)\in \mathcal{C}$ and define the symmetric two--tensor
\begin{equation}
F_{ij}=\,t \left( \theta f_{ij} + a \Delta f g_{ij} + b f_i f_j + c  |\nabla f|^2 g_{ij} + d e^{(p-1)f} g_{ij} \right).
\end{equation}
Observe that
\begin{equation*}
\frac{F_{ij; k} f_k}{t} = \theta f_{ijk} f_k + a (\Delta f)_k f_k g_{ij} + b (f_{ik} f_k f_j +  f_{jk} f_k f_i) + 2 c f_{lk} f_l f_k g_{ij} + d (p-1) e^{(p-1)f} |\nabla f|^2 g_{ij}
\end{equation*}
and, by a simple computation,
\begin{equation}\label{eq-1}
(\partial_t - \Delta) F_{ij} =  \frac{F_{ij}}{t} +  2 \frac{F_{ij; k} f_k}{t} + t Q_{ij},
\end{equation}
where
\begin{align}\label{eq:Qijdef}
\begin{split} 
Q_{ij} &= (p-1) e^{(p-1)f} \frac{F_{ij}}{t} + (p-1) \left[ b + (p-1) \theta \right] e^{(p-1)f} f_i f_j\\
&\quad+(p-1)\left[c+a(p-1)-dp\right]e^{(p-1)f}|\nabla f|^2 g_{ij}\\
&\quad+2(\theta-b)f_{ik}f_{jk}+2(a-c)|\nabla^2 f|^2 g_{ij}.
\end{split}
\end{align}
Define $\ringg{f}_{ij}:=f_{ij}-(\Delta f /n) g_{ij}$ to be the tracefree part of the Hessian of $f$. Since
$$
|\ringg{f}_{ij}|^2=|\nabla^2 f|^2-\frac1n (\Delta f)^2,
$$
we have 
\begin{align}\label{eq:ineqint1}
\begin{split}
2(\theta-b)& f_{ik}f_{jk}+2(a-c)|\nabla^2 f|^2 g_{ij}\\
&= \frac{2(\theta-b)}{\theta^2}\left(\theta f_{ik}+a \Delta f g_{ik}\right)\left(\theta f_{jk}+a \Delta f g_{jk}\right)+2(a-c)|\nabla^2 f|^2 g_{ij}\\
&\quad-\frac{4a(\theta-b)}{\theta}\Delta f f_{ij}-\frac{2a^2(\theta-b)}{\theta^2}(\Delta f)^2 g_{ij}\\
&= \frac{2(\theta-b)}{\theta^2}\left(\theta f_{ik}+a \Delta f g_{ik}\right)\left(\theta f_{jk}+a \Delta f g_{jk}\right)+2(a-c)|\ringg{f}_{ij}|^2 g_{ij}\\
&\quad-\frac{4a(\theta-b)}{\theta}\Delta f \ringg{f}_{ij}-\frac{2}{n\theta^2}\left[a(\theta-b)(2\theta+na)-(a-c)\theta^2\right](\Delta f)^2 g_{ij}\\
&= \frac{2(\theta-b)}{\theta^2}\Bigl(\frac{F_{ik}}{t}-b f_i f_k - c  |\nabla f|^2 g_{ik} - d e^{(p-1)f} g_{ik}\Bigr)\\
&\hspace{1.85cm}\cdot\Big(\frac{F_{jk}}{t}-b f_j f_k - c  |\nabla f|^2 g_{jk} - d e^{(p-1)f} g_{jk}\Big)+2(a-c)|\ringg{f}_{ij}|^2 g_{ij}\\
&\quad-\frac{4a(\theta-b)}{\theta}\Delta f \ringg{f}_{ij}+\frac{2}{n\theta^2}\left[(a-c)\theta^2-a(\theta-b)(2\theta+na)\right](\Delta f)^2 g_{ij}.
\end{split}
\end{align}
In order to estimate the last three terms, we use the algebraic inequality for traceless symmetric two--tensor
$$
\ringg{f}_{ij}\leq \ringg{\rho} g_{ij} \leq \sqrt{\frac{n-1}{n}} |\ringg{f}_{ij}|g_{ij},
$$
where $\ringg{\rho}=\max\{|\lambda_i|: \ \lambda_i \text{ eigenvalue of } \ringg{f}_{i,j}\}$ denotes the spectral radius of $\ringg{f}_{ij}$. By Young's inequality, for every $K>0$, one has
\begin{equation}\label{eq:ineqint2}
2 \Delta f \ringg{f}_{ij} \leq 2\sqrt{\frac{n-1}{n}} \Delta f|\ringg{f}_{ij}|g_{ij} \leq \sqrt{\frac{n-1}{n}}\Big[K(\Delta f)^2+\frac{1}{K}|\ringg{f}_{ij}|^2\Big]g_{ij}. 
\end{equation}
We set
$$
K := \sqrt{\frac{n-1}{n}} \frac{a(\theta-b)}{\theta(a-c)},
$$
where $K>0$, by the first two inequalities of~\eqref{eq:systineq} since $(a,b,c,d,\theta)\in \mathcal{C}$. Using inequality~\eqref{eq:ineqint2} to estimate the last three terms of equation~\eqref{eq:ineqint1} we achieve
\begin{align*}
2(a-c)&|\ringg{f}_{ij}|^2 g_{ij}-\frac{4a(\theta-b)}{\theta}\Delta f \ringg{f}_{ij}+\frac{2}{n\theta^2}\left[(a-c)\theta^2-a(\theta-b)(2\theta+na)\right](\Delta f)^2 g_{ij}\\
&\geq \frac{2}{n\theta^2(a-c)}\left\{(a-c)^2\theta^2-a(\theta-b)\left[(2\theta+na)(a-c)\right.\right.\\&\left.\left.\qquad +a(n-1)(\theta-b)\right]\right\}(\Delta f)^2 g_{ij} \geq 0,
\end{align*}
where the last inequality follows from the third inequality in~\eqref{eq:systineq} and since $(a,b,c,d,\theta)\in \mathcal{C}$. The computation above yields
\begin{align}
\begin{split}
2&(\theta-b)f_{ik}f_{jk}+2(a-c)|\nabla^2 f|^2 g_{ij}\\
&\geq \frac{2(\theta-b)}{\theta^2}\Bigl(\frac{F_{ik}}{t}-b f_i f_k - c  |\nabla f|^2 g_{ik} - d e^{(p-1)f} g_{ik}\Bigr)\\
&\qquad \qquad \quad 
\cdot \Bigl(\frac{F_{jk}}{t}-b f_j f_k - c  |\nabla f|^2 g_{jk} - d e^{(p-1)f} g_{jk}\Bigr)\\
&= \frac{2(\theta-b)}{\theta^2} \Big[\frac{F^2_{ij}}{t^2}-\frac{b}{t}(F_{ij}f_jf_k+F_{jk}f_if_k)-\frac{2c}{t}|\nabla f|^2F_{ij}\\
&\hspace{2.3cm} -\frac{2d}{t}e^{(p-1)f}F_{ij}+b(b+2c)|\nabla f|^2f_if_j \\
&\hspace{2.3cm}+2bd\, e^{(p-1)f}f_if_j+c^2|\nabla f|^4g_{ij}\\
&\hspace{2.3cm}+2cd\,e^{(p-1)f}|\nabla f|^2g_{ij}+d^2e^{2(p-1)f}g_{ij}\Big].\label{eq:est}
\end{split}
\end{align}
Now we claim that there exists $\varepsilon>0$ such that if $F_{ij}\le 0$ then
\begin{equation}\label{eq:claim}
\theta^2 Q_{ij} \geq \frac{F^2_{ij}}{\eps t^2},
\end{equation}
where $F^2_{ij} = F_{ik} \, F_{kj}$.
Thus, let us suppose that $F_{ij}\le 0$. Combining the estimate~\eqref{eq:est} with the definition of $Q_{ij}$ given in~\eqref{eq:Qijdef} we get
\begin{align*} 
	\theta^2 Q_{ij} &\geq (p-1)\theta^2 e^{(p-1)f} \frac{F_{ij}}{t} + (p-1)\theta^2 \left[ b + (p-1) \theta \right] e^{(p-1)f} f_i f_j\\
	&\quad+(p-1)\theta^2\left[c+a(p-1)-dp\right]e^{(p-1)f}|\nabla f|^2 g_{ij}\\
	&\quad+2(\theta-b)\Big[\frac{F^2_{ij}}{t^2}-\frac{b}{t}(F_{ij}f_jf_k+F_{jk}f_if_k)-\frac{2c}{t}|\nabla f|^2F_{ij}\\
	&\hspace{2.7cm}-\frac{2d}{t}e^{(p-1)f}F_{ij}+b(b+2c)|\nabla f|^2f_if_j \\
	&\hspace{2.7cm}+2bd\, e^{(p-1)f}f_if_j+c^2|\nabla f|^4g_{ij}\\
	&\hspace{2.7cm}+2cd\,e^{(p-1)f}|\nabla f|^2g_{ij}+d^2e^{2(p-1)f}g_{ij}\Big]\\
	&\geq 2(\theta-b)\frac{F^2_{ij}}{t^2}+\left[(p-1)\theta^2-4d(\theta-b)\right]e^{(p-1)f} \frac{F_{ij}}{t}\\
	&\quad+\left\{(p-1)\theta^2\left[c+a(p-1)-dp\right]+4cd(\theta-b) \right\}e^{(p-1)f}|\nabla f|^2 g_{ij}\\
	&\quad +2d^2(\theta-b)e^{2(p-1)f}g_{ij}.
\end{align*}
Recalling that for $(a,b,c,d,\theta)\in \mathcal{C}$ there holds $(\theta-b)>0$, we only have to show that 
\begin{equation*}
\begin{cases}
(p-1)\theta^2-4d(\theta-b) \le 0\\
(p-1)\theta^2\left[c+a(p-1)-dp\right]+4cd(\theta-b) \ge 0
\end{cases}
\end{equation*}
Setting $x=p-1 \ge 0$, we can recast the previous inequalities as
\begin{equation}\label{eq:systemx}
\begin{cases}
x\theta^2-4d(\theta-b)\le  0\\
(d-a)\theta^2 x^2+(d-c)\theta^2x-4cd(\theta-b)\le 0
\end{cases}
\end{equation}
Being $4cd(\theta-b)>0$, $(d-a)\theta^2>0$ and $(d-c)\theta^2>0$, equation~\eqref{eq:square} admits two solutions, of which only one is positive. Hence, defining $G_1(b,d,\theta)$ and $G_2(a,b,c,d,\theta)$ as in the statement of the Theorem and $G(a,b,c,d,\theta)$ as in~\eqref{def:G}, the fact that $0 \le x \le G(a,b,c,d,\theta)$ implies that $x$ satisfies both the inequalities in~\eqref{eq:systemx}. Thus, equation~\eqref{eq:claim} is satisfied with $\varepsilon=\frac{1}{2(\theta-b)}$.\\
The application of maximum principle (see for example Theorem C.1.3 in~\cite{Manlib} or Lemma 8.2  in~\cite{hamilton2}) concludes the proof.
\qed

\section{Proof of Proposition~\ref{prop:exist}}\label{pfexist}

First of all, observe that if $(a,b,c,d,\theta)$ is admissible and $\lambda>0$ is a constant, then $(\lambda a, \lambda b, \lambda c, \lambda d, \lambda \theta)$ is still admissible. Thus, if we show that an admissible quintuple of parameters exists, then we have a cone of admissible parameters. Let us consider a quintuple of the form $(a,b,c,a,\theta)$ where $a-c>0$ and $\theta-b>0$. We want to find $a,b,c,\theta$ such that the third inequality of~\eqref{eq:systineq} is satisfied. To do this, let us set $a-c=\delta$ for some $\delta>0$. Then, the third inequality of~\eqref{eq:systineq} becomes
\begin{align*}
	\delta^2\theta^2&-\delta(\theta-b)[(2\theta+n\delta+nc)\delta+\delta(n-1)(\theta-b)+c(n-1)(\theta-b)]\\
	&-c(\theta-b)[(2\theta+n\delta+nc)\delta+\delta(n-1)(\theta-b)+c(n-1)(\theta-b)]\ge 0.
\end{align*}
Now let us assume that $\theta-b=c$ and $b=kc$ for some $k>0$ to achieve
\begin{equation*}
	-n\delta^3c+(k^2-3n)\delta^2c^2-(3n+2k)\delta c^3-c^4(n-1)\ge 0.
\end{equation*}
Being $c>0$, we can set $z=\frac{\delta}{c}$ and recast the previous inequality as
\begin{equation*}
	-nz^3+(k^2-3n)z^2-(3n+2k)z-(n-1)\ge 0.
\end{equation*}
Let us set
\begin{equation*}
	H(z,k):=-nz^3+(k^2-3n)z^2-(3n+2k)z-(n-1).
\end{equation*}
We want to find some suitable values for $z$ and $k$ such that $H(z,k)\ge 0$. First of all, let us determine some necessary conditions on $k$ in such a way that $H(\cdot,k)$ admits positive roots.
\begin{prop}
	There always exists a negative number $z_-<0$ such that $H(z_-,k)=0$. If there exist two positive numbers $z_1,z_2>0$ such that $H(z_i,k)=0$, then $k > \sqrt{3n}$.
\end{prop}
\begin{proof}
	Let us observe that $H(0,k)=-(n-1)<0$ and $\lim_{z \to -\infty}H(z,k)=+\infty$, hence we obviously have $z_-$. Moreover, let us observe that, by Descartes' rule of the signs, the maximum number of positive solutions of the equation $H(z,k)=0$ for fixed $k$ is $0$ if $k^2-3n\le 0$ and $2$ if $k^2-3n>0$, thus, if two positive solutions $z_1,z_2$ exist, then $k>\sqrt{3n}$.
\end{proof}
\noindent We can use the previous necessary condition to achieve a necessary and sufficient condition on the existence of two positive roots
\begin{prop}
	Fix $n \ge 2$. Then $H(z,k)$ admits two (possibly equal) positive roots if and only if $k \ge k(n)$, where
	\begin{equation*}
		k(n)=3\sqrt{n}\cos\Bigl(\frac{1}{3}\arccos\left({1}/{\sqrt{n}}\,\right)\Bigr).
	\end{equation*}
\end{prop}
\begin{proof}
	As we have shown before, a necessary condition to have two positive roots is that $k > \sqrt{3n}$. On the other hand, let us observe that
	\begin{equation*}
		H(-z,k)=nz^3+(k^2-3n)z^2+(3n+2k)z-(n-1)
	\end{equation*}
	hence, by Descartes' rule of the signs we know that $H(\cdot,z)$ admits at most one negative root. Moreover, since $H(0,k)=-(n-1)<0$ and $\lim_{z \to -\infty}H(z,k)=+\infty$, then we know that $H(z,k)$ admits exactly one negative root for any $k>\sqrt{3n}$. Thus, if $H(z,k)$ admits three real roots, two of them have to be positive.\\
	The discriminant of $H(\cdot,k)$ is given by
	\begin{align*}
		\Delta_H(k)&=-18n(n-1)(k^2-3n)(3n+2k)+4(n-1)(k^2-3n)^3\\
		&\quad+(k^2-3n)^2(3n+2k)^2-4n(3n+2k)^3-27n^2(n-1)^2\\
		&=4(1+k)^3n\Big(k^3-\frac{27}{4}kn-\frac{27}{4}n\Big).
	\end{align*}
	We have that $\Delta_H(k)\ge 0$ if and only if $P_1(k):=k^3-\frac{27}{4}kn-\frac{27}{4}n\ge 0$. $P_1(k)$ is a depressed cubic polynomial with $p_1(n)=q_1(n)=-\frac{27}{4}n$. By Descartes' rule of the signs and the fact that $P_1(0)=-\frac{27}{4}n<0$ we know that $P_1(k)$ always admits a unique positive root. The discriminant of $P_1(k)$ is given by
	\begin{equation*}
		\Delta_1(n)=-\Big(-\frac{27^3}{4^2}n^3+\frac{27^3}{4^2}n^2\Big)=\frac{27^3}{16}n^2(n-1)>0,
	\end{equation*}
	thus $P_1(k)$ admits three real roots. Since we are under the \textit{casus irreducibilis}, we have to provide trigonometric solutions to recognize what is the real solution we are interested in. To do this, we will use Vi\'ete's procedure. Consider the equation $P_1(k)=0$ and set $k=u\cos(\theta)$ to achieve
	\begin{equation*}
		u^3\cos^3(\theta)-\frac{27}{4}nu\cos(\theta)-\frac{27}{4}n=0.
\end{equation*}
	Multiplying everything by $\frac{4}{u^3}$ we get
	\begin{equation*}
		4\cos^3(\theta)-\frac{27}{u^2}n\cos(\theta)-\frac{27}{u^3}n=0.
	\end{equation*}
	Now set $\frac{27n}{u^2}=3$, that is to say $u=3\sqrt{n}$, to achieve
	\begin{equation*}
		4\cos^3(\theta)-3\cos(\theta)-\frac{1}{\sqrt{n}}=0.
	\end{equation*}
	Recalling that $4\cos^3(\theta)-3\cos(\theta)=\cos(3\theta)$, we get
	\begin{equation*}
		\cos(3\theta)={1}/{\sqrt{n}}
	\end{equation*}
	and then
	\begin{equation*}
		\theta=\frac{1}{3}\arccos\left({1}/{\sqrt{n}}\,\right)+\frac{2\pi j}{3},
	\end{equation*}
where $j=0,1,2$. Finally, we get the three real roots of $P_1(k)$ as
\begin{equation*}
	k_j=3\sqrt{n}\cos\Bigl(\frac{1}{3}\arccos\left({1}/{\sqrt{n}}\,\right)+\frac{2\pi j}{3}\Bigr), \ j=0,1,2.
\end{equation*}
However, we know that $P_1(k)$ admits only one positive root. Being 
\begin{equation*}
	0 \le \frac{1}{3}\arccos\left({1}/{\sqrt{n}}\,\right)\le \frac{\pi}{3},
\end{equation*}
we have that $k_0>0$ is the solution we are searching for. Thus, let us relabel
\begin{equation*}
	k(n):=3\sqrt{n}\cos\Bigl(\frac{1}{3}\arccos\left({1}/{\sqrt{n}}\,\right)\Bigr)
\end{equation*}
to conclude that $P_1(k)\ge 0$ if and only if $k \ge k(n)$.
\end{proof}
\noindent Now we can conclude the proof of Proposition~\ref{prop:exist}. Indeed, let us consider $k>k(n)$, in such a way that $\Delta_H(k)>0$. For such fixed $k$, $H(\cdot,k)$ admits two positive roots $0<z_1<z_2$. Consider any $z \in [z_1,z_2]$. Setting, without loss of generality, $c=1$, we then know that the quintuple $(z+1,k,1,z+1,k+1)$ is admissible.

\section{Proof of Proposition~\ref{prop:lower}} \label{pflower}

To prove Proposition~\ref{prop:lower}, we want to exhibit an admissible quintuple $(a,b,c,d,\theta)$ such that $G(a,b,c,d,\theta)$ can be explicitly calculated. To do this, let us consider again a quintuple of the form $(z+1,k,1,z+1,k+1)$ where $k \ge k(n)$. Moreover, let us consider $z(k,n)$ to be a positive local maximum point of $H(z,k)$, for fixed $k \ge k(n)$, such that $H(z(k,n),k) \ge 0$. If $k>k(n)$, then this maximum always exists, by a simple application of Rolle's theorem on the interval $[z_1,z_2]$, together with the fact that both $z_1,z_2$ are simple roots. If $k=k(n)$, then $z_1=z_2$ and it coincides with such local maximum of the polynomial $H(z,k(n))$. Let us evaluate it explicitly.
\begin{prop}\label{prop:critpoint}
	For fixed $n \ge 2$ and $k >\sqrt{3n}$, $H(z,k)$ admits two (eventually equal) critical points if and only if these critical points are positive and $k \ge k_0(n)$ where
	\begin{equation*}
		k_0(n)=2\sqrt{2n}\cos\Big[\frac{1}{3}\arccos\Big(\frac{3}{2\sqrt{2n}}\Big)\Big].
	\end{equation*}
	If $k>k_0(n)$, the local maximum point is given by
	\begin{equation*}
		z(k,n)=\frac{k^2-3n+\sqrt{k^4-6nk^2-6nk}}{3n}.
	\end{equation*}
	In particular, $k_0(n) \le k(n)$ and, if $k \ge k(n)$, there holds $H(z(k,n),k)\ge 0$.
\end{prop}
\begin{proof}
	Let us first observe that
	\begin{equation}\label{poli1}
		H'(z,k):=\frac{\partial}{\partial \, z}H(z,k)=-3nz^2+2(k^2-3n)z-(3n+2k).
	\end{equation}
	Since $k>\sqrt{3n}$, Descartes' rule of the signs tells us that, if the solutions exist, they must be positive. Hence we have only to show that the discriminant is nonnegative. Let us determine the discriminant of the polynomial~\eqref{poli1}:
	\begin{equation}\label{Dezk}
		\Delta_z(k)=4((k^2-3n)^2-3n(3n+2k))=4k(k^3-6nk-6n).
	\end{equation}
	Being $k>0$, $\Delta_z(k)\ge 0$ if and only if
	\begin{equation*}
		P_2(k):=k^3-6nk-6n\ge 0.
	\end{equation*}
	Let us first show that this polynomial admits a unique positive root. Observe that $P_2(0)=-6n<0$ and $\lim_{k \to +\infty}P_2(k)=+\infty$, thus $P_2(k)$ admits a positive root $k_0(n)$. Moreover, Descartes' rule of the signs tells us that $P_2(k)$ admits at most one positive root, hence $k_0(n)$ is the unique positive root. \\
	Now let us determine $k_0(n)$. First of all, let us observe that, since $P_2(k)$ is a depressed cubic polynomial, its discriminant is given by
	\begin{equation*}
		\Delta_k=-(-4\cdot 6^3n^3+27 \cdot 6^2n^2)=108n^2(8n-9)>0
	\end{equation*}
	since $n \ge 2$. Thus we know $P_2(k)$ admits three different roots. Since we are under the \textit{casus irreducibilis}, we have to provide trigonometric solutions to recognize what is the real solution we are interested in. Arguing again by Vi\'{e}te's procedure and selecting the unique positive root, we get
	\begin{equation*}
		k_0(n)=2\sqrt{2n}\cos\Big[\frac{1}{3}\arccos\Big(\frac{3}{2\sqrt{2n}}\Big)\Big].
	\end{equation*}
	Hence, as $k > k_0(n)$, we can find two solutions to equation $H'(z,k)=0$. Being $H(\cdot,k)$ a polynomial with $H(0,k)=-(n-1)$ and $\lim_{z \to -\infty}H(z,k)=+\infty$, we know that $H(z,k)$ is decreasing as $z \le 0$. Thus, the first critical point has to be a local minimum and the second critical point a local maximum. Therefore, writing explicitly  the second solution of $H'(z,k)=0$, we get
	\begin{equation}\label{eq:zkn}
		z(k,n)=\frac{k^2-3n+\sqrt{k^4-6nk^2-6nk}}{3n}.
	\end{equation}
Finally, observe that, by Rolle's theorem, if $k>k(n)$, then the interval $[z_1,z_2]$ has to admit one of the two critical points. Moreover, since $H(z,k)>0$ as $z \in (z_1,z_2)$ and $H(z_1,k)=H(z_2,k)=0$, then such critical point is the local maximum and $H(z(k,n),k)$ is positive. If $k=k(n)$, then $z_1=z_2$ is a double root and $H(z_1,k(n))=H(z_2,k(n))=0$. Being $H(z,k(n))\le 0$ for any $z \ge 0$ (since the other simple root is negative), we have that $z_1=z_2=z(k(n),n)$ and $H(z(k(n),n),k(n))=0$. This also obviously implies $k_0(n)< k(n)$.
\end{proof}
\noindent Next, we want to evaluate $G$ on the quintuple $(z(k,n)+1,k,1,z(k,n)+1,k+1)$ as $k \ge k(n)$. To do this, we first need to exploit a simple property of $z(k,n)$.
\begin{prop}
	For any $k \ge k(n)$ there holds $z(k,n)\ge 2$.
\end{prop}
\begin{proof}
	By equation~\eqref{eq:zkn}, we have that $z(k,n)\ge 2$ if and only if
	\begin{equation}\label{eq:ineq21}
	\sqrt{k^4-6nk^2-6nk}\ge 9n-k^2.
	\end{equation}
	Being $k \ge k(n)$, we already know that the quantity under the square root is nonnegative. On the other hand, if $k\ge 3\sqrt{n}$, then inequality~\eqref{eq:ineq21} holds true. Let us consider $k(n)\le k <3\sqrt{n}$. Then inequality~\eqref{eq:ineq21} is equivalent to
	\begin{equation*}
	k^4-6nk^2-6nk\ge 81n^2-18nk^2+k^4
	\end{equation*}
	that is to say
	\begin{equation*}
	4k^2-2k-27n\ge 0.
	\end{equation*}
	Since $k(n)>0$, the previous inequality is verified as
	\begin{equation*}
	k \ge \frac{1+\sqrt{1+108n}}{4}=:k_1(n)>\frac{1}{2}.
	\end{equation*}
	To conclude the proof, we have to show that $k_1(n)\le k(n)$. Being $k_0(n)\le k(n)$, this is obvious if $k_1(n)\le k_0(n)$, thus let us suppose $k_1(n)>k_0(n)$.\\
	Let us consider the function $g(k):=H(z(k,n),k)$ for $k \ge k_0(n)$. Being $z(k,n)$ a local maximum of $H$, there holds $\frac{\partial \, H}{\partial \, z}(z(k,n),k)=0$. Thus, we have
	\begin{equation*}
	g'(k)=2kz^2(k,n)-2z(k,n)=2z(k,n)(kz(k,n)-1).
	\end{equation*}
	In particular $g'(k)\ge 0$ if and only if $kz(k,n)\ge 1$. The latter holds if and only if
	\begin{equation}\label{eq:ineq22}
	k\sqrt{k^4-6nk^2-6nk}\ge 3n+3nk-k^3=-\frac{\Delta_z(k)}{4k}-3n(1+k),
	\end{equation}
	where $\Delta_z(k)$ is defined in equation~\eqref{Dezk}. Being $k \ge k_0(n)$, we have that $\Delta_z(k)\ge 0$ and then inequality~\eqref{eq:ineq22} is verified. This implies, in particular, that $g(k)$ is increasing as $k \ge k_0(n)$. Moreover, by definition of $k(n)$, there holds $g(k(n))\ge 0$. On the other hand, being $z(k_1(n),n)=2$ and $k_1(n)>\frac{1}{2}$, we get
	\begin{equation*}
	g(k_1(n))=H(2,k_1(n))=-2k_1(n)+1<0.
	\end{equation*}
	Hence, we have $k_1(n)<k(n)$, concluding the proof.
\end{proof}
\noindent With this property in mind, we can evaluate $G(z(k,n)+1,k,1,z(k,n)+1,k+1)$. To do this, let us first observe that
\begin{equation*}
G_1(k,z(k,n)+1,k+1)=\frac{4(z(k,n)+1)}{(k+1)^2}.
\end{equation*}
Concerning $G_2(z(k,n)+1,k,1,z(k,n)+1,k+1)$, it is the unique solution of
\begin{equation*}
z(k,n)(k+1)^2x-4(z(k,n)+1)=0
\end{equation*}
that is to say
\begin{equation*}
G_2(z(k,n)+1,k,1,z(k,n)+1,k+1)=\frac{4(z(k,n)+1)}{z(k,n)(k+1)^2}.
\end{equation*}
Hence, we have
\begin{equation*}
G(z(k,n)+1,k,1,z(k,n)+1,k+1)=\frac{4(z(k,n)+1)}{(k+1)^2}\min\Big\{1,\frac{1}{z(k,n)}\Big\}.
\end{equation*}
However, being $z(k,n)\ge 2$ by the previous proposition, there holds 
$$
\min\Bigl\{1,\frac{1}{z(k,n)}\Bigr\}=\frac{1}{z(k,n)}
$$
then
\begin{equation*}
G(z(k,n)+1,k,1,z(k,n)+1,k+1)=\frac{4}{(k+1)^2}\Big(1+\frac{1}{z(k,n)}\Big).
\end{equation*}
Now we want to optimize on $k \ge k(n)$. To do this, let us show the following Proposition.
\begin{prop}
	The function $k \mapsto z(k,n)$ is increasing on $[k(n),+\infty)$.
\end{prop}
\begin{proof}
	Just observe that
	\begin{align*}
	z'(k)&=\frac{1}{3n}\Big(2k+\frac{2k^3-6nk-3n}{\sqrt{k^4-6nk^2-6nk}}\Big)
	\\&=\frac{1}{3n}\Big(2k+\frac{\Delta_z(k)}{4k\sqrt{k^4-6nk^2-6nk}}+\frac{k^3+3n}{\sqrt{k^4-6nk^2-6nk}}\Big)\ge 0,
	\end{align*}
	being $\Delta_z(k)\ge 0$ by the fact that $k \ge k(n)\ge k_0(n)$.
\end{proof}
\noindent The previous proposition implies that the function $k \mapsto G(z(k,n)+1,k,1,z(k,n)+1,k+1)$ is decreasing as $k \ge k(n)$, thus it achieve its maximum value as $k=k(n)$. Setting
\begin{equation*}
\widetilde{G}(n):=G(z(k(n),n)+1,k(n),1,z(k(n),n)+1,k(n)+1), \, n \ge 0,
\end{equation*}
we conclude the proof of Proposition~\ref{prop:lower}.

\section{Matrix Harnack inequalities}\label{matharn}

We begin the section with the following technical lemma.

\begin{lem}
There exists a cone $\mathcal{C}'$ of an admissible quintuples $(a,b,c,d,\theta)$ such that $a=d$ and $b-a+c<0$.
\end{lem}

\begin{proof}
Let us consider $k \geq k(n)$ with $k(n)$ defined in Proposition~\ref{prop:lower} and let $(z(k,n)+1,k,1,z(k,n)+1,k+1) \in \mathcal{C}$ with $z(k,n)$ defined in~\eqref{eq:zkn}. The condition $b-a+c<0$ becomes $k<z(k,n)$, \textit{i.e.} 
$$
3nk-k^2+3n<\sqrt{k^4-6nk^2-6nk}.
$$
This is obviously true if 
$$
k>\frac{3n+\sqrt{9n^2+12n}}{2}.
$$
With this choice of $k$ we conclude the proof.
\end{proof}

\begin{rem}
It is easy to see that if $(a,b,c,d,\theta) \in \mathcal{C}$ then also $(a,b,c,a,\theta) \in \mathcal{C}$.
\end{rem}

We will now derive some Harnack--type inequalities as a consequence of Theorem~\ref{main}.

\begin{prop}\label{harn}
Let $(M,g)$ be an $n$--dimensional complete Riemannian manifold with nonnegative sectional curvatures as well as parallel Ricci tensor. Consider $(a,b,c,a,\theta) \in \mathcal{C'}$ and $1<p<1 + G(a.b,c,a,\theta)$. Let $u:M\times[0,T)\to\R$ be a classical positive solution of the equation $u_t = \Delta u + u^p$, then there exists $\varepsilon =\varepsilon(n,p,a,b,c,\theta)$ such that, given any $0<t_1<t_2\leq T$ and $x_1,x_2 \in M$, the following inequality holds
\begin{equation*}
u(x_1,t_1) \leq u(x_2,t_2) \Big(\frac{t_2}{t_1}\Big)^{1/\varepsilon} \exp(\psi(x_1,x_2,t_1,t_2)),
\end{equation*}
where 
\begin{align*}
\psi(x_1,x_2,t_1,t_2)&:= \inf_{\gamma \in \Gamma(x_1,x_2)} \int_{0}^{1} \Big[\frac{a \vert \dot{\gamma}(s)\vert^2}{4(a-b-c) (t_2-t_1)}\\
&\qquad\qquad\qquad \qquad + \frac{\theta (t_2-t_1)}{a} \rho(\gamma(s), (1-s)t_2 + st_1) \Big] \, ds,
\end{align*}
with $\Gamma(x_1,x_2)$ given by all the paths in $M$ parametrized by $[0,1]$ joining $x_2$ to $x_1$, $f= \log u$ and $\rho=\max \{ |\lambda_i |: \, \lambda_i \textit{ eigenvalue of } f_{ij} \}$. 
\end{prop}
\begin{proof}
By Theorem~\ref{main} we know that there exists $\varepsilon$ such that
\begin{equation}\label{harna0}
\theta f_{ij} + a \Delta f g_{ij} + b f_i f_j + c |\nabla f|^2 g_{ij} + a e^{(p-1)f} g_{ij}+ \frac{1}{\varepsilon t} g_{ij} \geq 0.
\end{equation}
Recalling that $f$ satisfies
\begin{equation*}
f_t = \Delta f + |\nabla f|^2 + e^{(p-1)f},
\end{equation*}
we get 
\begin{equation}\label{harna1}
-f_t \, g_{ij} \leq \frac{\theta}{a} f_{ij} + \frac{b}{a} f_i f_j -\frac{a-c}{a} |\nabla f|^2 g_{ij} + \frac{1}{a \varepsilon t} g_{ij}.
\end{equation}
Since $\rho$ is the spectral radius of $f_{ij}$ we get $f_{ij} \leq \rho g_{ij}$. On the other hand the matrix $f_i f_j$ has rank one and the only nonzero eigenvalue is $|\nabla f|^2$, then $f_i f_j \leq |\nabla f|^2 g_{ij}$. Plugging these inequalities into~\eqref{harna1} we obtain
\begin{equation}\label{harna2}
-f_t \, g_{ij} \leq \Big(\frac{\theta}{a} \rho -\frac{a-c-b}{a} |\nabla f|^2 + \frac{1}{a \varepsilon t} \Big) g_{ij}.
\end{equation}
Now let us consider any $\gamma \in \Gamma(x_1,x_2)$ as well as $\eta:[0,1]\to M\times[t_1,t_2]$ defined as $\eta(s)= (\gamma(s), (1-s)t_2+st_1)$. Evaluating~\eqref{harna2} in $\eta(s)$ and applying it to $\dot{\gamma}(s)$ we see that
\begin{equation}\label{harna3}
-f_t (\eta(s)) |\dot{\gamma}(s)|^2 \leq \Big(\frac{\theta}{a} \rho(\eta(s)) -\frac{a-c-b}{a} |\nabla f(\eta(s))|^2 + \frac{1}{a \varepsilon ((1-s)t_2+st_1)} \Big) |\dot{\gamma}(s)|^2.
\end{equation}
Noticing that $\eta(0)= (x_2, t_2)$ and $\eta(1)= (x_1, t_1)$, we have
\begin{align*}
f(x_1,t_1) - &f(x_2,t_2)  = \int_{0}^{1} \Big( \frac{d\,}{ds} f(\eta(s)) \Big) \, ds\\ 
&\,= \int_{0}^{1} \left[\langle \nabla f(\eta(s)), \dot{\gamma} (s) \rangle - (t_2-t_1) f_s (\eta(s)) \right] \, ds\\
&\, \leq \int_{0}^{1} \Big[ \vert \nabla f(\eta(s)) \vert \vert \dot{\gamma}\vert + (t_2-t_1) \Big(\frac{1}{a \varepsilon [(1-s)t_2+st_1]}\\ 
& \qquad\qquad\qquad\quad\qquad\qquad\qquad\qquad\left.-\frac{a-c-b}{a} \vert\nabla f(\eta(s)) \vert^2 + \frac{\theta}{a} \rho(\eta(s)) \right) \Big] \, ds\\
&\, = \int_{0}^{1} \frac{t_2-t_1}{a \varepsilon [(1-s)t_2+st_1]} \, ds + \int_{0}^{1} \Big[\vphantom{\frac{(a-c-b)(t_2-t_1)}{a}} \vert \nabla f (\eta(s))\vert \vert \dot{\gamma}(s)\vert\\
&\qquad\qquad\qquad\qquad - \frac{(a-c-b)(t_2-t_1)}{a} \vert\nabla f (\eta(s))\vert^2+ \frac{\theta(t_2-t_1)}{a} \rho(\eta(s)) \Big] \, ds\\
&\, \leq \frac{1}{a \varepsilon} \log \Big( \frac{t_2}{t_1}\Big) + \int_{0}^{1} \Big[\frac{a \vert \dot{\gamma}(s)\vert^2}{4(a-b-c) (t_2-t_1)} + (t_2-t_1) \frac{\theta}{a} \rho(\eta(s)) \Big] \, ds.
\end{align*}
Being $\gamma \in \Gamma(x_1,x_2)$ arbitrary, we conclude the proof.
\end{proof}

With the same strategy we can also prove the following variant.

\begin{prop}\label{harn2}
Let $(M,g)$ be an $n$--dimensional complete Riemannian manifold with nonnegative sectional curvatures as well as parallel Ricci tensor. Consider $(a,b,c,a,\theta) \in \mathcal{C'}$ and $1<p<1 + G(a.b,c,a,\theta)$. Let $u:M\times[0,T)\to\R$ a classical positive solution of the equation $u_t = \Delta u + u^p$, then there exists $\varepsilon =\varepsilon(n,p,a,b,c,\theta)$ such that, given any $0<t_1<t_2\leq T$ and $x_1,x_2 \in M$, the following inequality holds
\begin{equation*}
u(x_1,t_1) \leq u(x_2,t_2) \Big(\frac{t_2}{t_1}\Big)^{1/\varepsilon} \exp(\psi(x_1,x_2,t_1,t_2)),
\end{equation*}
where 
\begin{align*}
\psi(x_1,x_2,t_1,t_2)&:= \inf_{\gamma \in \Gamma(x_1,x_2)} \int_{0}^{1} \Big[\frac{na \vert \dot{\gamma}(s)\vert^2}{4(na-nb-nc+\theta) (t_2-t_1)}\\&\qquad\qquad\qquad\qquad+ \frac{\theta (t_2-t_1)}{a} \ringg{\rho}(\gamma(s), (1-s)t_2 + st_1) \Big] \, ds,
\end{align*}
with $\Gamma(x_1,x_2)$ given by all the paths in $M$ parametrized by $[0,1]$ joining $x_2$ to $x_1$, $f= \log u$, $\ringg{f}_{ij} = f_{ij}- \frac{\Delta f}{n} g_{ij} $ and $\ringg{\rho}=\max \{ |\lambda_i |: \, \lambda_i \textit{ eigenvalue of } \ringg{f}_{ij} \}$. 
\end{prop}

Finally, we also have another one coming from the scalar (trace) version of the Li \& Yau inequality~\eqref{MAIN2}. We underline that it holds up to an exponent $p$ lower than the analogous one for the Harnack inequality obtained in~\cite{cacama2}, moreover, this latter holds also without the hypothesis of parallel Ricci tensor.
  
\begin{prop}\label{harn3}
Let $(M,g)$ be an $n$--dimensional complete Riemannian manifold with nonnegative sectional curvatures as well as parallel Ricci tensor. Consider $(a,b,c,a,\theta) \in \mathcal{C}$ and $1<p<1 + G(a.b,c,a,\theta)$. Let $u:M\times[0,T)\to\R$ a classical positive solution of the equation $u_t = \Delta u + u^p$, then there exists $\varepsilon =\varepsilon(n,p,a,b,c,\theta)$ such that, given any $0<t_1<t_2\leq T$ and $x_1,x_2 \in M$, the following inequality holds
\begin{equation*}
u(x_1,t_1) \leq u(x_2,t_2) \Big(\frac{t_2}{t_1}\Big)^{1/\varepsilon} \exp(\psi(x_1,x_2,t_1,t_2)),
\end{equation*}
where 
\begin{align*}
\psi(x_1,x_2,t_1,t_2)&:= \inf_{\gamma \in \Gamma(x_1,x_2)} \int_{0}^{1} \Big[\frac{(\theta+na) \vert \dot{\gamma}(s)\vert^2}{4(na-nc+\theta-b) (t_2-t_1)}\\&\qquad\qquad\qquad\qquad - \frac{\theta}{\theta+na} u^{p-1}(\gamma(s), (1-s)t_2 + st_1) \Big] \, ds,
\end{align*}
with $\Gamma(x_1,x_2)$ given by all the paths in $M$ parametrized by $[0,1]$ joining $x_2$ to $x_1$. 
\end{prop}
\begin{proof}
Let $f=\log u$, by Theorem~\ref{main} we know that there exists $\varepsilon>0$ such that inequality~\eqref{MAIN2} holds, then the proof proceeds as in Proposition~\ref{harn}.
\end{proof}

\bibliographystyle{amsplain}
\bibliography{biblio}

\providecommand{\bysame}{\leavevmode\hbox to3em{\hrulefill}\thinspace}
\providecommand{\MR}{\relax\ifhmode\unskip\space\fi MR }
\providecommand{\MRhref}[2]{%
  \href{http://www.ams.org/mathscinet-getitem?mr=#1}{#2}
}
\providecommand{\href}[2]{#2}
\begin{thebibliography}{1}

\bibitem{cacama2}
D.~Castorina, G.~Catino, and C.~Mantegazza, \emph{Semilinear {L}i \& {Y}au
  inequalities}, in preparation.

\bibitem{hamilton2}
R.~S. Hamilton, \emph{Four--manifolds with positive curvature operator}, J.
  Diff. Geom. \textbf{24} (1986), no.~2, 153--179.

\bibitem{hamilton12}
\bysame, \emph{The {H}arnack estimate for the {R}icci flow}, J. Diff. Geom.
  \textbf{37} (1993), no.~1, 225--243.

\bibitem{hamilton7}
\bysame, \emph{A matrix {H}arnack estimate for the heat equation}, Comm. Anal.
  Geom. \textbf{1} (1993), no.~1, 113--126.

\bibitem{hamilton4}
\bysame, \emph{The {H}arnack estimate for the mean curvature flow}, J. Diff.
  Geom. \textbf{41} (1995), no.~1, 215--226.

\bibitem{liyau}
P.~Li and S.-T. Yau, \emph{On the parabolic kernel of the {S}chr\"odinger
  operator}, Acta Math. \textbf{156} (1986), no.~3--4, 153--201.

\bibitem{Manlib}
C.~Mantegazza, \emph{Lecture notes on mean curvature flow}, Progress in
  Mathematics, vol. 290, Birkh\"auser/Springer Basel AG, Basel, 2011.

\end{thebibliography}
\end{document}